\documentclass[11pt]{amsart}
\usepackage{amsmath, amsthm, amssymb}
\usepackage{geometry}
\usepackage{setspace}
\usepackage{array, booktabs}
\usepackage{enumitem}
\setlist[enumerate,1]{label={(\arabic*)},leftmargin=3em}

\theoremstyle{plain}
\newcounter{thmcount}[section]

\theoremstyle{definition}
\newtheorem{Definition}[thmcount]{Definition}
\theoremstyle{plain}
\newtheorem{Lemma}[thmcount]{Lemma}
\newtheorem{Theorem}[thmcount]{Theorem}

\newcommand{\cald}{\mathcal{D}}
\newcommand{\calg}{\mathcal{G}}
\newcommand{\band}{\overline{\mathcal{G}}}

\title[Equal Lagrange Numbers with Non-Isomorphic Band Graphs]{Equal Lagrange
Numbers with Non-Isomorphic Band Graphs\\
\large A counterexample to Schiffler's Problem 6.3 }
\author{Qiyue Tang \and Yizhi Zhang}
\date{}
\theoremstyle{definition}
\newtheorem{Example}[thmcount]{Example}
\begin{document}

\begin{abstract}
We consider the following question: if two lattice paths in the same set
$\cald(a,b)$ have the same Lagrange number, must their band graphs be
isomorphic? We exhibit two explicit lattice paths in $\cald(17,9)$ with the
same Lagrange number but non-isomorphic associated band graphs, thereby giving
a counterexample to this question.

\vspace{1em}
\noindent\textbf{Keywords:} lattice paths, snake graphs, band graphs,
continued fractions, Lagrange numbers.

\noindent\textbf{MSC(2020):} 06A07, 11J06, 05C70, 13F60.
\end{abstract}

\maketitle

\section{Introduction}

Ralf Schiffler studied the following class of lattice paths in
\cite{Schiffler}. Fix relatively prime positive integers $a,b$ with $0<b<a$.
Consider lattice paths from $(0,0)$ to $(a,b)$ whose steps are either rightward
or upward and which never cross above the diagonal joining $(0,0)$ to $(a,b)$.
Equivalently, every point $(x,y)$ visited by the path satisfies $ay\le bx$.
The set of all such paths is denoted by $\cald(a,b)$.

For each $\omega\in\cald(a,b)$, Schiffler constructed a snake graph
$\calg(\omega)$ and a band graph $\band(\omega)$. Snake graphs and band graphs
arise in cluster algebras from surfaces, where they describe Laurent expansions
of cluster variables and related basis elements. In this lattice-path model,
the square lattice may be viewed as the universal cover of the once-punctured
torus: lattice paths correspond to arcs on the torus, while band graphs are
related to the combinatorial expansions of the corresponding closed curves
\cite[Introduction]{Schiffler}. Thus the problem links graph models from
cluster algebra with continued fractions in number theory.

Schiffler further associated two numerical invariants to every
$\omega\in\cald(a,b)$. The first, $M(\omega)$, is the number of perfect matchings
of the snake graph $\calg(\omega)$ and is related to Markov numbers. The second,
$L(\omega)$, is the Lagrange number of the periodic continued fraction determined
by the band graph $\band(\omega)$ and is related to the Lagrange spectrum. These
two quantities define the matching order and the Lagrange order on lattice paths.

Against this background, Schiffler posed Problem~6.3: for two paths
$\omega,\omega'$ in the same set $\cald(a,b)$, does
\[
 L(\omega)=L(\omega')
 \quad\Longrightarrow\quad
 \band(\omega)\cong\band(\omega')?
\]
See \cite[Problem 6.3]{Schiffler}. In other words, with the endpoints and the
diagonal constraint fixed, does the Lagrange number determine the isomorphism
class of the band graph? The cited source gives an example of two paths with
equal Lagrange numbers and isomorphic band graphs, but that example does not
establish the implication for arbitrary paths.

We give a counterexample and hence answer Problem~6.3 in the negative.

\begin{Theorem}
\label{thm:intro-main}
The two lattice paths
\[
\begin{aligned}
\omega_1&=RRURRURRRRRRURURUURRURRRUU,\\
\omega_2&=RRURRRURRRRRURUURURRURRRUU
\end{aligned}
\]
belong to $\cald(17,9)$ and satisfy
\[
 L(\omega_1)=L(\omega_2),
 \qquad
 \band(\omega_1)\not\cong\band(\omega_2).
\]
Consequently, lattice paths with the same Lagrange number need not have
isomorphic band graphs.
\end{Theorem}

The remainder of the paper is organized as follows. Section~2 introduces
lattice paths, coefficient words, periodic continued fractions, Lagrange
numbers, and the criterion used for band-graph isomorphism. Section~3 gives the
complete computation for the counterexample. Appendix~A contains a certificate
program for checking the finite computational data.

\medskip
\noindent\textbf{Use of AI.}
GPT-5.6 sol was used to find the counterexample that disproves the conjecture
addressed in this paper. Both authors subsequently reviewed the entire
manuscript in full and take full responsibility for its correctness.

\section{Background}

We recall the basic concepts and results needed below; the underlying
constructions come from \cite{Schiffler}.

\begin{Definition}
\label{def:lattice-path}
Let $a,b$ be relatively prime positive integers with $0<b<a$. A \emph{lattice
path} from $(0,0)$ to $(a,b)$ is a word in the letters $R$ and $U$, where $R$
denotes a rightward unit step and $U$ an upward unit step; the word contains
exactly $a$ letters $R$ and $b$ letters $U$.

If every point $(x,y)$ visited by the path satisfies
\[
 ay\le bx,
\]
then the path is said to lie below the diagonal joining $(0,0)$ to $(a,b)$;
touching the diagonal is allowed. The set of all such paths is denoted by
$\mathcal{D}(a,b)$. We further set
\[
 \mathcal{D}
 =\bigcup_{\substack{0<b<a\\ \gcd(a,b)=1}}\mathcal{D}(a,b).
\]
\end{Definition}

\begin{Definition}
\label{def:snake-graph}
A \emph{tile} is the square planar graph with four vertices and four edges. A
\emph{snake graph} $\mathcal{G}$ is a connected planar graph obtained by gluing
a finite sequence of tiles
\[
G_1,G_2,\ldots,G_d
\]
in order. For every $1\le i<d$, the tile $G_{i+1}$ shares exactly one full edge
with $G_i$: either the north edge of $G_i$ is identified with the south edge of
$G_{i+1}$, or the east edge of $G_i$ is identified with the west edge of
$G_{i+1}$. A newly added tile does not overlap any earlier nonconsecutive tile.
The common edges of consecutive tiles are called \emph{interior edges}; all
remaining edges are \emph{boundary edges} \cite[Section 2.1]{Schiffler}.
\end{Definition}

\begin{Definition}
\label{def:band-graph}
A \emph{sign function} on a snake graph assigns $+$ or $-$ to each edge so that,
on every tile, the north and west edges have the same sign, the south and east
edges have the same sign, and the north and south edges have opposite signs.
Let the tiles of $\mathcal{G}$ be $G_1,\ldots,G_d$. Choose either the south or
west edge of $G_1$, and then choose an edge of the same sign among the north and
east edges of $G_d$. Identifying these two edges by their corresponding endpoints
produces a \emph{band graph} \cite[Section 2.2]{Schiffler}.

If the lengths of the maximal constant blocks in the cyclic sign sequence of a
band graph are $a_1,\ldots,a_n$ in order, we denote the graph by
\[
\mathcal{G}[\overline{a_1,\ldots,a_n}].
\]
Because a band graph has no distinguished starting tile or direction, this
cyclic sequence is considered up to cyclic shift and reversal.
\end{Definition}

\begin{Definition}
\label{def:coefficient-word}
Let $\omega=x_1\cdots x_L$ be a lattice path in $\mathcal{D}$, with
$x_i\in\{R,U\}$. Its \emph{coefficient word}, denoted by $c(\omega)$, is formed
as follows. First write $2$. Then examine the adjacent pairs $x_ix_{i+1}$ from
left to right for $1\le i\le L-1$. Append $1,1$ if
$x_ix_{i+1}\in\{RR,UU\}$, and append $2$ if
$x_ix_{i+1}\in\{RU,UR\}$.

If $c(\omega)=c_1\cdots c_n$, then the band graph associated with $\omega$ is
\[
\overline{\mathcal{G}}(\omega)
=\mathcal{G}[\overline{c_1,\ldots,c_n}],
\]
in agreement with the construction obtained by placing tiles along the path and
identifying the first and last boundary edges \cite[Definition 4.1]{Schiffler}.
\end{Definition}

\begin{Example}
For \(\omega = RRU\), the rule above yields \(c(\omega) = 2112\).
\end{Example}

A \emph{finite continued fraction} is a function
\[
    [a_1, a_2, \dots, a_n] = a_1 + \cfrac{1}{a_2 + \cfrac{1}{\ddots + \cfrac{1}{a_n}}}
\]
of \(n\) positive integers \(a_1, \dots, a_n\). Similarly, an \emph{infinite continued fraction} is a function
\[
    [a_1, a_2, \dots] = a_1 + \cfrac{1}{a_2 + \cfrac{1}{a_3 + \cfrac{1}{\ddots}}}
\]
of infinitely many positive integers \(a_1, a_2, \dots\). We use the notation \([\overline{a_1, \dots, a_n}]\) for the periodic continued fraction, where \([\overline{a_1, \dots, a_n}] = [a_1, \dots, a_n, a_1, \dots, a_n, \dots]\).

\begin{Definition}
\label{def:lagrange-number}
Given a real number $\alpha$, its \emph{Lagrange number} $L(\alpha)$ is
defined as the supremum of all real numbers $\ell$ for which there exist
infinitely many rational numbers $\frac{p}{q}$ such that
\[
  \left|\alpha-\frac{p}{q}\right|<\frac{1}{\ell\,q^2}.
\]
\end{Definition}

For a lattice path $\omega$ with coefficient word
$c(\omega)=c_1c_2\cdots c_n$, Schiffler \cite[Definition 4.3]{Schiffler}
defines
\[
  L(\omega)=\text{Lagrange number of }[\overline{c_1,c_2,\ldots,c_n}],
\]
and \cite[Remark 4.4]{Schiffler} records the formula
\[
  L(\omega)=\max(\alpha-\alpha'),
\]
where $\alpha$ runs over all cyclic shifts of
$[\overline{c_1,c_2,\ldots,c_n}]$ and $\alpha'$ is the conjugate of $\alpha$;
see also \cite[Proposition 1.29]{Aigner}.  

\begin{Lemma}
\label{lem:matrix-formula}
Let $\omega$ be a lattice path with coefficient word
$c(\omega)=c_1c_2\cdots c_n$. For each positive integer $d$, put
\[
  A(d)=\begin{pmatrix}d&1\\1&0\end{pmatrix},
\]
and for $k=1,\ldots,n$, with indices taken modulo $n$, set
\[
  P_k=A(c_k)A(c_{k+1})\cdots A(c_{k+n-1})
     =\begin{pmatrix}p_k&r_k\\q_k&s_k\end{pmatrix}.
\]
Then the following hold.
\begin{enumerate}
\item $\det P_k=(-1)^n$ for every $k$.
\item $P_{k+1}=A(c_k)^{-1}P_kA(c_k)$; in particular the trace
  $T=\operatorname{tr}(P_k)=p_k+s_k$ is independent of $k$.
\item Let $\alpha_k=[\overline{c_k,c_{k+1},\ldots,c_{k+n-1}}]$ be the cyclic
  shift of $[\overline{c_1,\ldots,c_n}]$ starting at $c_k$, and let
  $\alpha_k'$ be its conjugate.  Then
  \[
    \alpha_k-\alpha_k'=\frac{\sqrt{T^2-4(-1)^n}}{|q_k|}.
  \]
\item Consequently,
  \[
    L(\omega)=L([\overline{c_1,\ldots,c_n}])
             =\frac{\sqrt{T^2-4(-1)^n}}{\min_k |q_k|}.
  \]
  If every $q_k>0$ and $n$ is even, this becomes
  \[
    L(\omega)^2=\frac{T^2-4}{(\min_k q_k)^2}.
  \]
\end{enumerate}
\end{Lemma}

\begin{proof}
Since $\det A(d)=-1$, (1) follows from the multiplicativity of the determinant.
For (2), $c_{k+n}=c_k$, so
\[
  P_{k+1}=A(c_k)^{-1}P_kA(c_k),
\]
and the trace is invariant under conjugation.

The fixed points of the fractional linear transformation
$x\mapsto(p_kx+r_k)/(q_kx+s_k)$ are the roots of
\[
  q_kx^2+(s_k-p_k)x-r_k=0,
\]
namely $\alpha_k$ and $\alpha_k'$. Their difference satisfies
\[
  (\alpha_k-\alpha_k')^2
  =\frac{(p_k-s_k)^2+4q_kr_k}{q_k^2}
  =\frac{(p_k+s_k)^2-4(p_ks_k-q_kr_k)}{q_k^2}
  =\frac{T^2-4(-1)^n}{q_k^2}.
\]
As $\alpha_k>\alpha_k'$, taking the positive square root gives (3), and (4)
follows from Definition~\ref{def:lagrange-number} since
\[
  \max_k(\alpha_k-\alpha_k')
  =\frac{\sqrt{T^2-4(-1)^n}}{\min_k |q_k|}.
\]
\end{proof}

\begin{Definition}
\label{def:cyclic-equivalence}
Let \(\omega\) and \(\upsilon\) be lattice paths, and let
\[
c(\omega)=c_1c_2\cdots c_n,\qquad c(\upsilon)=d_1d_2\cdots d_n
\]
be their associated coefficient words in \(\{1,2\}^n\), where indices are understood modulo \(n\) (so that \(c_{i+n}=c_i\) and \(d_{i+n}=d_i\) for all \(i\)).

Define the cyclic shift operator \(T\) by
\[
T(c_1,\dots,c_n)=(c_2,\dots,c_n,c_1),
\]
and the reversal operator \(R\) by
\[
R(c_1,\dots,c_n)=(c_n,\dots,c_2,c_1).
\]
Two coefficient words \(c,d\in\{1,2\}^n\) are said to be \emph{cyclically equivalent}, denoted \(c\sim d\), if there exists an integer \(k\) with \(0\le k<n\) such that
\[
d=T^k(c)\quad\text{or}\quad d=T^k(R(c)).
\]

We now define isomorphism between the associated band graphs. The band graphs of
\(\omega\) and \(\upsilon\), denoted respectively by
\[
\overline{\mathcal{G}}(\omega)\quad\text{and}\quad
\overline{\mathcal{G}}(\upsilon),
\]
are called \emph{isomorphic}, written
\[
\overline{\mathcal{G}}(\omega)\cong \overline{\mathcal{G}}(\upsilon),
\]
if and only if their coefficient words satisfy
\[
c(\omega)\sim c(\upsilon).
\]
\end{Definition}

\begin{Lemma}
\label{lem:triple-two-invariant}
Let \(\omega\) be a lattice path with coefficient word
\[
c(\omega)=c_1c_2\cdots c_n \in \{1,2\}^n,
\]
where indices are taken modulo \(n\). For each maximal cyclic block of consecutive \(2\)'s in \(c(\omega)\), let its length be \(L\). Suppose the lengths of all such blocks are
\[
L_1, L_2, \dots, L_m \qquad (m\ge 0),
\]
where \(m=0\) if there are no \(2\)'s. Define the integer
\[
\mathcal{N}\bigl(c(\omega)\bigr) := \sum_{j=1}^{m} \max(0,\, L_j - 2).
\]
Equivalently, \(\mathcal{N}\bigl(c(\omega)\bigr)\) counts the number of occurrences of the contiguous subword \(222\) in the cyclic word \(c(\omega)\), with multiplicity (so a block of length \(L\) contributes exactly \(L-2\) such occurrences).

Then \(\mathcal{N}\) is invariant under the cyclic equivalence relation \(\sim\)
from Definition~\ref{def:cyclic-equivalence}. Consequently, if two band graphs
are isomorphic, i.e.
\[
\overline{\mathcal{G}}(\omega) \cong \overline{\mathcal{G}}(\upsilon),
\]
then their coefficient words satisfy
\[
\mathcal{N}\bigl(c(\omega)\bigr) = \mathcal{N}\bigl(c(\upsilon)\bigr).
\]
Thus \(\mathcal{N}\) provides a necessary condition for band-graph isomorphism.
\end{Lemma}

\begin{proof}
Let \(c=c_1\cdots c_n\in\{1,2\}^n\). Consider the two generating operations of the equivalence relation \(\sim\):

\begin{enumerate}
\item \textit{Cyclic shift.} For \(T(c)=(c_2,\dots,c_n,c_1)\), the cyclic order of the runs is merely rotated; the multiset of their lengths \(\{L_1,\dots,L_m\}\) remains unchanged. Hence
\[
\mathcal{N}(T(c))=\mathcal{N}(c).
\]

\item \textit{Reversal.} For \(R(c)=(c_n,\dots,c_2,c_1)\), the cyclic order of the runs is reversed, but the multiset of block lengths is again preserved. Therefore
\[
\mathcal{N}(R(c))=\mathcal{N}(c).
\]
\end{enumerate}

By iteration, for every integer \(k\) with \(0\le k<n\),
\[
\mathcal{N}(T^k(c))=\mathcal{N}(c)
\quad\text{and}\quad
\mathcal{N}(T^k(R(c)))=\mathcal{N}(c).
\]
Thus, whenever \(c\sim d\), we have \(\mathcal{N}(c)=\mathcal{N}(d)\).

By Definition~\ref{def:cyclic-equivalence},
\(\overline{\mathcal{G}}(\omega)\cong \overline{\mathcal{G}}(\upsilon)\)
implies \(c(\omega)\sim c(\upsilon)\). Substituting these coefficient words
gives the desired equality
\[
\mathcal{N}\bigl(c(\omega)\bigr)=\mathcal{N}\bigl(c(\upsilon)\bigr).
\]
Therefore the equality of \(\mathcal{N}\)-values is a necessary condition for
band-graph isomorphism.
\end{proof}
\section{The counterexample in $\cald(17,9)$}
We now present the explicit counterexample and verify all required properties in a single coherent flow.
Consider the two lattice paths
\[
\omega_1=RRURRURRRRRRURURUURRURRRUU,
\]
\[
\omega_2=RRURRRURRRRRURUURURRURRRUU.
\]
Each word contains exactly $17$ letters $R$ and $9$ letters $U$, and $\gcd(17,9)=1$, so both paths go from $(0,0)$ to $(17,9)$.

We begin by checking the diagonal condition \(17y \le 9x\) at every point visited by the lattice path. The table below lists, for each point \((x,y)\) along the path (including the starting and ending points), the value \(17y-9x\). All values are \(\le 0\), with equality only at the two endpoints \((0,0)\) and \((17,9)\); hence both paths stay below the diagonal, i.e. \(\omega_1,\omega_2\in\mathcal{D}(17,9)\).

\noindent
\begingroup
\centering
\small
\begin{tabular}{c|cc|cc}
$i$ & $\omega_1$: $(x,y)$ & $17y-9x$ & $\omega_2$: $(x,y)$ & $17y-9x$\\ \hline
0 & (0,0) & 0 & (0,0) & 0\\
1 & (1,0) & $-9$ & (1,0) & $-9$\\
2 & (2,0) & $-18$ & (2,0) & $-18$\\
3 & (2,1) & $-1$ & (2,1) & $-1$\\
4 & (3,1) & $-10$ & (3,1) & $-10$\\
5 & (4,1) & $-19$ & (4,1) & $-19$\\
6 & (4,2) & $-2$ & (5,1) & $-28$\\
7 & (5,2) & $-11$ & (5,2) & $-11$\\
8 & (6,2) & $-20$ & (6,2) & $-20$\\
9 & (7,2) & $-29$ & (7,2) & $-29$\\
10 & (8,2) & $-38$ & (8,2) & $-38$\\
11 & (9,2) & $-47$ & (9,2) & $-47$\\
12 & (10,2) & $-56$ & (10,2) & $-56$\\
13 & (10,3) & $-39$ & (10,3) & $-39$\\
14 & (11,3) & $-48$ & (11,3) & $-48$\\
15 & (11,4) & $-31$ & (11,4) & $-31$\\
16 & (12,4) & $-40$ & (11,5) & $-14$\\
17 & (12,5) & $-23$ & (12,5) & $-23$\\
18 & (12,6) & $-6$ & (12,6) & $-6$\\
19 & (13,6) & $-15$ & (13,6) & $-15$\\
20 & (14,6) & $-24$ & (14,6) & $-24$\\
21 & (14,7) & $-7$ & (14,7) & $-7$\\
22 & (15,7) & $-16$ & (15,7) & $-16$\\
23 & (16,7) & $-25$ & (16,7) & $-25$\\
24 & (17,7) & $-34$ & (17,7) & $-34$\\
25 & (17,8) & $-17$ & (17,8) & $-17$\\
26 & (17,9) & 0 & (17,9) & 0\\
\end{tabular}
\par
\endgroup

\vspace{1.5ex}
\noindent Next, applying Definition~\ref{def:coefficient-word} to the two paths
yields the coefficient words
\[
c(\omega_1)=21122112211111111112222211211221111211;
\]
\[
c(\omega_2)=21122111122111111112221122211221111211.
\]
Both words have the even length $n=38$. This parity is essential: by
Lemma~\ref{lem:matrix-formula}, we have $\det P_k=(-1)^{38}=1$ for every cyclic
product $P_k$, so the formula in that lemma simplifies to
\[
L(\omega)^2=\frac{T^2-4}{(\min_k q_k)^2},
\]
where $T$ is the common trace and $q_k$ is the bottom-left entry of $P_k$.

For each of the two coefficient words we compute the 38 products
$P_k=A(c_k)\cdots A(c_{k+37})$, with indices taken modulo 38 and
$A(d)=\begin{pmatrix}d&1\\1&0\end{pmatrix}$. The trace is constant for each word; remarkably, it is the same for both:
\[
T=19974212955.
\]
The following table gives the bottom-left entries $q_k$ for all cuts
$k=1,\ldots,38$; the minimum at $k=36$ is highlighted.

\begin{center}
\footnotesize
\begin{tabular}{c|rr|c|rr}
$k$ & $q_k(\omega_1)$ & $q_k(\omega_2)$ & $k$ & $q_k(\omega_1)$ & $q_k(\omega_2)$\\\hline
1 & 6307169825 & 6306933953 & 20 & 6586998175 & 6577139161\\
2 & 9546673265 & 9548244695 & 21 & 7145420917 & 7214212165\\
3 & 9335043271 & 9332017717 & 22 & 7036448485 & 6645965863\\
4 & 6730429813 & 6739387909 & 23 & 7131860335 & 9418370671\\
5 & 6715897465 & 6656215087 & 24 & 6668361667 & 9417447665\\
6 & 9407705011 & 9747881827 & 25 & 9353941825 & 6647811875\\
7 & 9401349785 & 8716516475 & 26 & 9540102431 & 7202213087\\
8 & 6728607917 & 8718945791 & 27 & 6296040455 & 6645441605\\
9 & 6647811875 & 9743023195 & 28 & 9551948213 & 9431669285\\
10 & 9805329995 & 6668361667 & 29 & 9330250261 & 9389927431\\
11 & 8599744591 & 6656791165 & 30 & 6739436359 & 6728925313\\
12 & 9058982683 & 9800875705 & 31 & 6658029055 & 6659568985\\
13 & 8886853811 & 8603240771 & 32 & 9737286781 & 9736709071\\
14 & 8944002335 & 9052061033 & 33 & 8737875173 & 8738106005\\
15 & 8944685635 & 8903235181 & 34 & 8656852271 & 8656775117\\
16 & 8885487211 & 8900892475 & 35 & 9899332585 & 9899370847\\
17 & 9062399183 & 9056746445 & \textbf{36} & \textbf{6252914545} & \textbf{6252914545}\\
18 & 8590861691 & 8591527241 & 37 & 9494215027 & 9494175505\\
19 & 9828562195 & 9831330883 & 38 & 9467087387 & 9467165801\\
\end{tabular}
\end{center}

In both columns the minimum is attained \textbf{exactly once}, at $k=36$, and
the two minima are equal:
\[
q_{\min}=6252914545.
\]
At that cut the products are
\[
P_{1,36}=\begin{pmatrix}16138741633&9899332585\\
6252914545&3835471322\end{pmatrix},
\qquad
P_{2,36}=\begin{pmatrix}16138722187&9899370847\\
6252914545&3835490768\end{pmatrix}.
\]
Both have trace
\[
T=16138741633+3835471322
 =16138722187+3835490768
 =19974212955.
\]
as expected.

Since $n=38$ is even and $q_{\min}$ is the same for both words,
Lemma~\ref{lem:matrix-formula} gives
\[
L(\omega_1)^2=L(\omega_2)^2
=\frac{T^2-4}{q_{\min}^2}
=\frac{398969183171689832021}{39098940307072557025}.
\]
Both Lagrange numbers are positive by definition, hence $L(\omega_1)=L(\omega_2)$.

It remains to show that the band graphs are not isomorphic. By
Definition~\ref{def:cyclic-equivalence}, isomorphism of the band graphs is equivalent to
cyclic equivalence of their coefficient words under shifts and reversal. To
rule this out, we use the invariant from Lemma~\ref{lem:triple-two-invariant}:
the number $\mathcal{N}$ of cyclic occurrences of the contiguous subword $222$,
counted with multiplicity within each maximal block of consecutive 2's.

In
\[
c(\omega_1)=21122112211111111112222211211221111211,
\]
the factor $222$ occurs exactly 3 times cyclically: the block $22222$ contains three consecutive $222$'s. Hence
\[
\mathcal{N}(c(\omega_1))=3.
\]
In
\[
c(\omega_2)=21122111122111111112221122211221111211,
\]
the factor $222$ occurs exactly 2 times: $c(\omega_2)$ has two runs of three consecutive 2's , each contributing one cyclic occurrence of $222$. Thus
\[
\mathcal{N}(c(\omega_2))=2.
\]
Since these counts differ, Lemma~\ref{lem:triple-two-invariant} implies that the
two coefficient words are not cyclically equivalent. Therefore, by
Definition~\ref{def:cyclic-equivalence}, the band graphs $\band(\omega_1)$ and
$\band(\omega_2)$ are not isomorphic as abstract graphs.

This completes the counterexample. The program \texttt{exact\_certificate.py}
in Appendix~A uses only the Python standard library and exact integer arithmetic.
It independently checks the finite computational data used in this section:
the validity of the two lattice paths, their coefficient words, the common trace
and the smallest lower-left entry of the cyclic matrix products, the exact value
of the square of the Lagrange number, and the inequivalence of the two coefficient
words under cyclic shift and reversal. The program does not replace the proofs
of the general lemmas in the text.

\appendix
\section{The certificate program}

\begin{verbatim}
"""Exact certificate for the D(17,9) counterexample (self-contained)."""

from fractions import Fraction

SIGMA_1 = "RRURRURRRRRRURURUURRURRRUU"
SIGMA_2 = "RRURRRURRRRRURUURURRURRRUU"
C_1 = "21122112211111111112222211211221111211"
C_2 = "21122111122111111112221122211221111211"

def counts(w):
    return w.count("R"), w.count("U")

def prefix_condition(w, a, b):
    x = y = 0
    for ch in w:
        if ch == "R":
            x += 1
        else:
            y += 1
        if a * y > b * x:
            return False
    return True

def coefficient_word(w):
    out = ["2"]
    for i in range(len(w) - 1):
        out.append("11" if w[i] == w[i + 1] else "2")
    return "".join(out)

def mat_mul(X, Y):
    return (
        (X[0][0]*Y[0][0] + X[0][1]*Y[1][0],
         X[0][0]*Y[0][1] + X[0][1]*Y[1][1]),
        (X[1][0]*Y[0][0] + X[1][1]*Y[1][0],
         X[1][0]*Y[0][1] + X[1][1]*Y[1][1]),
    )

def A(d):
    return ((d, 1), (1, 0))

def cyclic_products(c):
    m = len(c)
    mats = [A(int(ch)) for ch in c]
    out = []
    for k in range(m):
        X = ((1, 0), (0, 1))
        for j in range(m):
            X = mat_mul(X, mats[(k + j) % m])
        out.append(X)
    return out

def canonical_cyclic(w):
    r = w[::-1]
    return min(min(x[i:] + x[:i] for i in range(len(x))) for x in (w, r))

def cyclic_factor_count(w, pat):
    n = len(w)
    return sum(1 for i in range(n) if (w * 2)[i:i + len(pat)] == pat)

def check(cond, msg):
    assert cond, "FAILED: " + msg
    print("ok:", msg)

def main():
    for w, c in ((SIGMA_1, C_1), (SIGMA_2, C_2)):
        check(counts(w) == (17, 9), "17 R and 9 U")
        check(prefix_condition(w, 17, 9), "stays below the diagonal")
        check(coefficient_word(w) == c, "coefficient word matches")
    P1 = cyclic_products(C_1)
    P2 = cyclic_products(C_2)
    T1 = {p[0][0] + p[1][1] for p in P1}
    T2 = {p[0][0] + p[1][1] for p in P2}
    check(len(T1) == 1 and len(T2) == 1, "trace is cut-independent")
    T = next(iter(T1))
    check(T == next(iter(T2)) == 19974212955, "common trace T")
    q1 = [p[1][0] for p in P1]
    q2 = [p[1][0] for p in P2]
    check(min(q1) == min(q2) == 6252914545, "common minimal q")
    # Python index 35 corresponds to the mathematical cut k = 36.
    check(q1.index(min(q1)) == 35 and q2.index(min(q2)) == 35,
          "unique minimum at mathematical cut 36")
    check(P1[35] == ((16138741633, 9899332585), (6252914545, 3835471322)),
          "P1 at mathematical cut 36")
    check(P2[35] == ((16138722187, 9899370847), (6252914545, 3835490768)),
          "P2 at mathematical cut 36")
    L2 = Fraction(T * T - 4, min(q1) * min(q1))
    check(L2 == Fraction(398969183171689832021, 39098940307072557025),
          "exact L^2")
    check(canonical_cyclic(C_1) != canonical_cyclic(C_2),
          "cyclic words inequivalent up to reversal")
    check(cyclic_factor_count(C_1, "222") != cyclic_factor_count(C_2, "222"),
          "222-counts differ")
    print()
    print("EXACT CERTIFICATE PASSED")

if __name__ == "__main__":
    main()
\end{verbatim}

\end{document}